\documentclass[12pt]{amsart}

\usepackage{amssymb, amsmath, xcolor}
\usepackage{import}
\usepackage{tikz}
\tikzset{filled/.style={minimum width=5pt,inner sep=0pt,circle,fill=black}}
\usetikzlibrary {arrows.meta}
\usepackage{ytableau}
\usepackage{subcaption}
\usepackage{mathtools}
\usepackage{todonotes}
\usepackage[utf8]{inputenc}
\usepackage{float}

\newtheorem{theorem}{Theorem}[section]
\newtheorem{lemma}[theorem]{Lemma}
\newtheorem{corollary}[theorem]{Corollary}
\newtheorem{conjecture}[theorem]{Conjecture}

\theoremstyle{definition}
\newtheorem{definition}[theorem]{Definition}

\theoremstyle{remark}
\newtheorem{remark}[theorem]{Remark}

\numberwithin{equation}{section}
\numberwithin{figure}{section}

\renewcommand{\mod}{\operatorname{mod}}
\newcommand{\YT}{\operatorname{YT}}
\newcommand{\M}{\operatorname{M}}
\newcommand{\R}{\operatorname{R}}
\newcommand{\mM}{\operatorname{mM}}
\newcommand{\Out}{\operatorname{Out}}
\newcommand{\In}{\operatorname{In}}
\newcommand{\N}{\mathbb{N}}

\title[Young Tableaux Via Minors]{Young Tableau Reconstruction Via Minors}
\author[Erickson, Herden, Meddaugh, Sepanski, $\ldots$]{William Q. Erickson, Daniel Herden, Jonathan Meddaugh, Mark R. Sepanski, Cordell Hammon, Jasmin Mohn, Indalecio Ruiz-Bolanos}
\address{
All authors:
Department of Mathematics,
Baylor University,
Sid Richardson Building,
1410 S.~4th Street,
Waco, TX 76706, USA}
\email{will\_erickson@baylor.edu, daniel\_herden@baylor.edu,  jonathan\_meddaugh@baylor.edu, mark\_sepanski@baylor.edu}
\date{\today}

\begin{document}

\begin{abstract}
    The tableau reconstruction problem, posed by Monks (2009), asks the following.  Starting with a standard Young tableau $T$, a 1-minor of $T$ is a tableau obtained by first deleting any cell of $T$, and then performing jeu de taquin slides to fill the resulting gap. This can be iterated to arrive at the set of $k$-minors of $T$.  The problem is this: given $k$, what are the values of $n$ such that every tableau of size $n$ can be reconstructed from its set of $k$-minors?  For $k=1$, the problem was recently solved by Cain and Lehtonen.  In this paper, we solve the problem for $k=2$, proving the sharp lower bound $n \geq 8$.  In the case of multisets of $k$-minors, we also give a lower bound for arbitrary $k$, as a first step toward a sharp bound in the general multiset case.  \end{abstract}

\maketitle
\tableofcontents
\ytableausetup{centertableaux}

\section{Introduction}

In \cite{Monks2009}, Monks solved the \emph{partition reconstruction problem} posed by Pretzel and Siemons in \cite{Pretzel}: given a positive integer $k$, determine the values of $n$ such that each partition of $n$ can be reconstructed from its set of $k$-minors (i.e., the partitions obtained by deleting $k$ cells from the corresponding Young diagram).  This problem was originally motivated by the representation theory of the symmetric group $S_n$: Monks's solution thus determines when the character of an $S_n$-module can be recovered from its restriction to stabilizer subgroups of $S_n$.

In turn, Monks posed an analogue of the problem by asking the same question for standard Young tableaux \cite[Section 4.3]{Monks2009}.  These are Young diagrams whose cells are given a total ordering such that the cells increase along rows and columns.  With this added structure, Monks defined a natural analogue for $k$-minors in terms of Sch\"utzenberger's \emph{jeu de taquin} procedure.

Specifically, the \emph{tableau reconstruction problem} is the following.  Starting with a tableau $T$, we choose any cell to delete, and then perform the necessary sequence of jeu de taquin ``slides'' to fill the resulting gap until we obtain a new Young tableau, one cell smaller than the original; see Section \ref{sec:prelim} for details. We call this new tableau a \emph{$1$-minor} of $T$.  The \emph{set of $1$-minors} of $T$ collects the resulting $1$-minors for all possible choices of initial cells to delete.  The \emph{set of $k$-minors} of $T$ is then defined recursively to be the union of the sets of $1$-minors of the $(k-1)$-minors of $T$.  The problem is this: given $k$, what are the values of $n$ such that every tableau of size $n$ can be reconstructed from its set of $k$-minors?  For $k=1$, the problem was recently solved by Cain and Lehtonen \cite{CainLehtonen2022}, who proved reconstructibility for the sharp bound $n \geq 5$.  As the authors of the paper observe, their methods are quite specific to the $k=1$ case.

It is worth observing that for $k=1$, the process proposed by Monks for deleting cells can be interpreted as a generalization of Schüt\-zen\-ber\-ger's \emph{dual promotion operator} on tableaux \cite{Schutzenberger1972}.  In particular, the dual promotion  operator yields the $1$-minor obtained by deleting the upper-left cell in the tableau (assuming we ignore the largest cell in the result; see Remark~\ref{rem:promotion}).  Incidentally, by invoking the inverse of the dual promotion operator, we are able to give a highly efficient algorithm to reconstruct $T$ directly from a certain distinguished $1$-minor, as opposed to the inductive procedure in~\cite{CainLehtonen2022}; see Remark~\ref{rem:alternate proof M1}.

Our main result in this paper is a solution to the tableau reconstruction problem for $k = 2$.  In this case, we prove that $n \geq 8$ is the sharp lower bound for reconstructibility from $2$-minors, see Theorem~\ref{thm: k = 2 sharp}. For general $k$, we have found some evidence to suggest that the lower bound is $n \geq k^2 + 2k$, which arises as the size of a $(k+1) \times (k+1)$ tableau with one cell removed (see Figure~\ref{figure: k+1 square minus lower right corner} and Conjecture~\ref{conj:k^2 + 2k}).  We also consider the problem in terms of \emph{multisets} of $k$-minors; in this case, we are able to formulate a lower cubic bound for all $k$ (see Theorem~\ref{thm: exp lower bound}).  The software calculations we performed to verify our results have also led to a conjecture for multisets of minors; see Conjecture~\ref{conj:k+4}, where the bound on $n$ appears to be linear in $k$.

\section*{Acknowledgement}
The authors thank Daniel Bossaller for helpful suggestions on this project.

\section{Preliminaries}
\label{sec:prelim}

In this paper, we write $\N$ for $\mathbb{Z}_{\geq 1}$ and will use interval notation restricted to $\N$. For example, we will write $[1, n]$ for $\{1, 2, \ldots, n\}$.

For $n \in \N$, a \emph{partition} of $n$ is a weakly decreasing finite sequence $\lambda = (\lambda_1, \ldots, \lambda_m)$ of positive integers such that $\sum_i \lambda_i = n$.
 We call $n$ the \emph{size} of $\lambda$.
  The \emph{Young diagram} of shape $\lambda$ is a left-aligned array of cells with $\lambda_h$ boxes in the $h$th row, counting from the top. As an example, the Young diagram of shape $(5, 5, 4, 2, 1, 1)$ is given in Figure~\ref{figure: yd 554211}.

\begin{figure}[H]
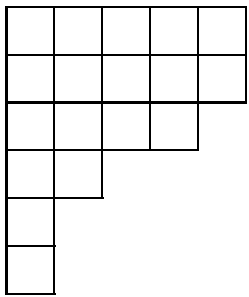

    \centering
    \ydiagram{5,5,4,2,1,1}
    \caption{Young diagram of shape $(5, 5, 4, 2, 1, 1)$}
    \label{figure: yd 554211}
\end{figure}

Let $\lambda$ be a partition of $n$, and suppose $N\geq n$. A \emph{standard Young tableau} of shape $\lambda$, over the alphabet $[1, N]$, is obtained from the Young diagram of $\lambda$ by filling the cells with distinct elements of $[1, N]$ so that entries in each row and column are strictly increasing. In this paper, we will simply use the word \emph{tableau} to mean a standard Young tableau. We write $\YT(n,N)$ for the set of tableaux of size $n$ over $[1,N]$.  In this paper, usually $N = n$; in this case, we simply write $\YT(n)$
for the set of tableaux of size $n$ over $[1,n]$.
For example, Figure \ref{figure: yt 554211} gives an example of a tableau of shape $(5, 5, 4, 2, 1, 1)$ over $[1,18]$. We often identify a cell with its entry; we also say that a cell is \emph{labeled} by its entry.

\begin{figure}[H]
    \centering
    \begin{ytableau}
        1 & 2 & 3 & 4 & 5 \\
        6 & 7 & 8 & 9 & 18 \\
        10 & 11 & 12 & 17 \\
        13 & 16 \\
        14 \\
        15
    \end{ytableau}
    \caption{A standard Young tableau}
    \label{figure: yt 554211}
\end{figure}

Given a tableau $T\in\YT(n)$, an \emph{outer corner} (\emph{OC} for short) is a cell of $T$ which is both the right end of a row and the bottom end of a column of $T$. For example, in Figure \ref{figure: yt 554211}, the OCs are $15$, $16$, $17$, and~$18$.

If $R$ is a collection of $k$ cells of $T$ we write $|R| = k$. Thus $|T| = n$.

If $m$ is an entry of a cell $c$ of $T$, we define its
\emph{outer area}, denoted by either $\Out(m)$ or $\Out(c)$, to be the collection of all cells of $T$ that are to the right of $m$ or below $m$. Figure \ref{figure: out example} gives an example of $\Out(m)$.

\begin{figure}[H]
    \centering
    \ytableaushort
        {\none, \none, \none, \none, \none \none m, \none, \none, \none}
        * {5,5,4,3,3,2,1,1}
        * [*(lightgray)]{3+2,3+2,3+1,3+0,3+0,0+2,0+1,0+1}
    \caption{$\Out(m)$}
    \label{figure: out example}
\end{figure}

We also define the \emph{inner area} of $m$, written as either $\In(m)$ or $\In(c)$, to be the set of cells that are both weakly left of $m$ and weakly above~$m$.
 Figure \ref{figure: in example} gives an example of $\In(m)$.  Note that all cells in $\In(m)$ have entries that are less than or equal to $m$.

\begin{figure}[H]
    \centering
    \ytableaushort
        {\none, \none, \none, \none, \none \none m, \none, \none, \none}
        * {5,5,4,3,3,2,1,1}
        * [*(lightgray)]{3,3,3,3,3,0,0,0}
    \caption{$\In(m)$}
    \label{figure: in example}
\end{figure}

There is a natural way of deleting cells from a tableau using the process known as \emph{jeu de taquin} (introduced in  \cite{Schützenberger1977} in the context of rectifying semistandard skew tableaux). To describe the process, begin with $T \in \YT(n)$ and a cell filled with the entry $m$. Begin by deleting this cell, leaving an empty space. Next, iterate the following procedure until it terminates: if there exists a cell either directly to the right or directly below the empty space, slide the cell with the smaller entry into the position of the empty space. This terminates when there are no cells directly to the right or below the current empty space; in other words, the process necessarily terminates at an OC. Finally, subtract $1$ from each entry larger than $m$. The resulting tableau is an element of $\YT(n-1)$ and is denoted by $T - m$. The jeu de taquin works the same way for $T \in \YT(n,N)$ and produces an element of $\YT(n-1,N-1)$.
An example is given in Figure \ref{figure: cell deletion example}.

\begin{figure}[H]
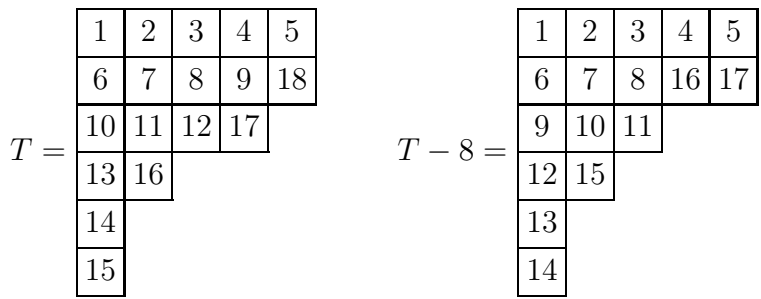

    \centering
    $T =$
    \begin{ytableau}
        1 & 2 & 3 & 4 & 5 \\
        6 & 7 & 8 & 9 & 18 \\
        10 & 11 & 12 & 17 \\
        13 & 16 \\
        14 \\
        15
    \end{ytableau}
    \;\;\;\;\;\;\;
    $T - 8 =$
    \begin{ytableau}
        1 & 2 & 3 & 4 & 5 \\
        6 & 7 & 8 & 16 & 17 \\
        9 & 10 & 11  \\
        12 & 15 \\
        13 \\
        14
    \end{ytableau}
    \caption{Cell deletion via jeu de taquin}
    \label{figure: cell deletion example}
\end{figure}

For $T \in \YT(n)$ and $k \in \N$, a $k$\emph{-minor} of $T$ is a tableau formed by iteratively deleting $k$ cells from $T$ via jeu de taquin. We write
$\M_k(T)$ for the \emph{set of all $k$-minors} of $T$ and $\mM_k(T)$ for the \emph{multiset of all $k$-minors} of $T$.  Clearly if $k \geq n$, then $\M_k(T) = \varnothing$, so we assume $k < n$.
If $T'\in \M_k(T)$ contains an OC which is also an OC of $T$ (possibly with a different entry), then we call the cell a \emph{surviving outer corner} (\emph{SOC} for short).

It is known
\cite[Thm.~2.5]{Monks2009} that the shape of $T\in\YT(n)$ can be recovered from the shapes of the tableaux in $\M_k(T)$ when
\begin{equation} \label{equation: Monks shape determination bound}
    n \geq k^2 + 2k.
\end{equation}
It is known that this bound is sharp, although in general, for a given $k$, there exist some sporadic tableaux $T$ with $n<k^2 + 2k$ for which the shape of $T$ can still be recovered from $\M_k(T)$.

\begin{remark}
\label{rem:promotion}
As mentioned in the introduction, when the initial deleted entry $m$ is $1$, the process described above is essentially the dual promotion operator $\partial^*$ introduced in~\cite{Schutzenberger1972}.
(Elsewhere in the literature, confusingly, this operator is also called the promotion operator, or even the demotion operator.)
To be precise, $T-1 = \partial^*(T) - n$ for all $T \in \YT(n)$. There is also an inverse, denoted by $\partial$.
The operators $\partial$ and $\partial^*$ have been the subject of more recent research in \cite{Shimozono, Stanley, Rhoades, Ahlbach}, among others.
\end{remark}

\begin{remark}
\label{rem:alternate proof M1}
The invertibility of $\partial^*$ allows for a highly efficient algorithm to reconstruct $T \in \YT(n)$ from $\M_1(T)$ for all $n \geq 5$.
(This differs from the inductive method described in ~\cite{CainLehtonen2022}.)
Since $\M_1(T)$ determines the shape of $T$ for $n \geq 5$, one can determine the non-surviving OC in any minor; then since $\partial^*$ is invertible, in order to recover $T$, it suffices to identify $T-1 \in \M_1(T)$.
Define the \emph{initial string} of a tableau to be the maximal set of cells labeled $1, \ldots, \ell$ which form an unbroken row or column. We have $\ell \ge 2$. One easily checks the following:
\begin{itemize}
\item If $\M_1(T)$ contains a minor whose initial string is properly contained in the initial strings of all other minors, then this minor is $T-1 =T-2=\ldots =T-\ell$, where $\ell$ denotes the length of the initial string of $T$, and we are done. This case occurs whenever $\ell \geq 3$.

\item If there is no minimal initial string, then $\ell = 2$ and so the location of the entry $2$ in $T-1$ $( = T-2)$ differs from all other minors.
Hence as long as $|\M_1(T)\,| \geq 3$, we can immediately identify $T-1$ and we are done.
If $|\M_1(T)\,| = 2$, then $T$ must have the shape $(n-1,1)$ or its transpose, with $T-1$ ${( = T-2)}$ being the unique minor consisting of one single row or one single column, respectively.
\end{itemize}
\end{remark}

\section{General Identification Lemmas}

In this section we give lemmas that can help reconstruct parts of tableaux from their minors.

\begin{lemma} \label{lemmma: construct In(m) when m in same place}
    Let $n,k\in\N$ and $T\in\YT(n)$. Suppose $m$ is a known entry of a cell $c$ in $T$ that is an OC. If there exists $T'\in\M_k(T)$ in which the cell $c$ exists and is labeled by $m$, then all cells of $T$ with entries less than $m$ are reconstructible from $\M_k(T)$. In particular, $\In(m)$ is reconstructible from $\M_k(T)$.
\end{lemma}

 \begin{proof}
    First note that, since $c$ is an OC, no cell may slide into its position during cell deletion.
    Now, in the construction of $T'$, if any cell had been deleted that was less than $m$, then the value of $m$ would have been reduced in $T'$. Since this did not happen, only cells larger than $m$ were deleted.
    Since every entry of $\In(m)$ in $T$ is less than $m$, only cells in $\Out(m)$ that were larger than $m$ were deleted. Any change resulting from deleting such cells is confined to cells in $\Out(m)$ that are greater than $m$. In particular, this means that the entries of $T$ that started out less than $m$ are all unchanged in $T'$ and therefore recoverable.
 \end{proof}

Slightly more generally, the same proof gives the following result.

\begin{lemma} \label{lemmma: general construct In(m) when m in same place and nothing could slide}
    Let $n,k\in\N$ and $T\in\YT(n)$. Suppose $m$ is a known entry of a cell $c$ in $T$ and it is known that no other cell can slide to $c$ and retain the same entry $m$ after $k$ deletions. If there exists $T'\in\M_k(T)$ in which the cell $c$ exists and is labeled by $m$, then all cells of $T$ with entries less than $m$ are reconstructible from $\M_k(T)$. In particular, $\In(m)$ is reconstructible from $M_k(T)$.
\end{lemma}


\begin{definition}
    Let $n \in \N$ and $T \in \YT(n)$. Write $$\R_n T$$ for the tableau in $\YT(n-1)$ obtained by removing the cell with label $n$ from $T$. (Throughout the paper, ``remove'' carries this obvious sense, while ``delete'' refers to the jeu de taquin process. Note, however, $R_n T = T - n$.)
 More generally for $d \in [1, n]$, write $\R_{[d, n]} T$ for the tableau $\R_{d}(\R_{d + 1}(\ldots (\R_n T)\ldots))$. We also allow $\R_{[d, n]}$ to be applied to a set of tableaux by applying $\R_{[d, n]}$ to each tableau in the set.
\end{definition}

In \cite[Lemma 3.3]{CainLehtonen2022}, it is shown that
\begin{equation} \label{equation: removal of n and M1}
    \M_1(\R_n T) = \R_{n-1} \M_1(T)
\end{equation}
for $T \in \YT(n)$. This result generalizes.

\begin{lemma} \label{lemma: removal of top elements and Mk}
    Let $n,k \in \N$, let $d \in [k+1,n]$, and let $T \in \YT(n)$. Then
        \[ \M_k (\R_{[d, n]} T) = \R_{[d-k, n-k]} \M_k (T). \]
\end{lemma}

\begin{proof}
    By iterating, it suffices to show that
    \[ \M_k (\R_{n} T) = \R_{n-k} \M_k (T). \]
    We show this equality via induction on $k$.
    The base case of $k = 1$ is already done in \eqref{equation: removal of n and M1}. Assuming the result is known up to $k - 1$, we simply calculate that
    \begin{align*}
        \M_k(\R_n T) &= \bigcup_{T' \in \M_{k-1}(\R_n T)} \M_1( T' ) \\
                    &= \bigcup_{T' \in \R_{n-k+1}\M_{k-1}(T)} \M_1( T' ) \\
                    &= \bigcup_{T'' \in \M_{k-1}(T)} \M_1( \R_{n-k+1} T'' ) \\
                    &= \bigcup_{T'' \in \M_{k-1}(T)} \R_{n-k} \M_1( T'' ) \\
                    &= \R_{n-k} \left(\bigcup_{T'' \in \M_{k-1}(T)} \M_1( T'' )\right)
                    = \R_{n-k} \M_k(T).\qedhere
    \end{align*}
\end{proof}

\begin{lemma}
\label{lemma:location of k2 + 2k through n}
    Let $n,k \in \N$ and $T \in \YT(n)$. If $n \geq k^2 + 2k + 1$, then the location in $T$ of each entry in $[(k+1)^2, n]$ is determined by $\M_k(T)$.
\end{lemma}

\begin{proof}
    Recall from \eqref{equation: Monks shape determination bound} that the shape of $T$ is determined by (the shapes in) $\M_k(T)$ when $n \geq k^2 + 2k$.
    When $n - 1 \geq k^2 + 2k$, the shape of $\R_n T$ is determined by $\M_k (\R_n T)$ which, in turn, is the same as $\R_{n-k} \M_k(T)$,
    by Lemma \ref{lemma: removal of top elements and Mk}. Since taking the complement of the shape of $\R_n T$ in the shape of $T$ gives us the location of $n$ in $T$, it follows that $\M_k(T)$ determines the location of $n$ when $n \geq k^2 + 2k + 1$.

    Next, when $n - 2 \geq k^2 + 2k$, the shape of $\R_{[n-1, \: n]} T$ is determined by $\M_k (\R_{[n-1, n]} T)$ which, in turn, is the same as $\R_{[n-k-1, \: n-k]} \M_k(T)$. Again by comparing the shapes of $\R_n T$ and $\R_{[n-1,\: n]} T$, it follows that $\M_k(T)$ determines the location of $n-1$ when $n \geq k^2 + 2k + 2$.

    Straightforward induction demonstrates that the location of $[m, n]$ is determined by $\M_k(T)$ when $n \geq k^2 + 2k + 1 + n - m.$ In particular, this is minimally satisfied by $m = k^2 + 2k + 1$ when $n \geq k^2 + 2k + 1.$
\end{proof}

\section{General Cubic Lower Bound}

\begin{theorem} \label{thm: exp lower bound}
    Let $n,k \in \N$ and $T \in \YT(n)$. The multiset $\mM_k(T)$ determines $T$ when
    \begin{equation} \label{eq:4.1}
    2 \binom{n - k^2 -2k}{k} > \binom{n}{k}.
    \end{equation}
    In particular, this is satisfied when
    \[ n > k^2 + 3k - 1+\frac{k^2 + 2k}{2^{\frac{1}{k}} - 1}, \]
    which gives
    \[ n > \frac{k^3+2k^2}{\ln 2}+\frac {k^2}2 +2k-1 + \frac{\ln 2}{12} (k+2) \]
    as a cubic lower bound.
\end{theorem}

\begin{proof}
    From Theorem \ref{lemma:location of k2 + 2k through n}, we know that the location in $T$ of each entry in $[(k+1)^2, n]$ is determined by $\M_k(T)$ when $n \geq k^2 + 2k + 1$. It remains to find the location of $1, 2, \ldots, k^2 + 2k$.

    For $m \in [1, k^2 + 2k]$, the location of $m$ remains fixed in any $k$-minor of $T$ when the deleted cells are chosen from $[m + 1, n]$. Thus there are at least $(n - m) (n - m - 1) \ldots (n - m - k + 1)$
    minors in $\mM_k(T)$ with $m$ fixed. As there are $n (n-1) \ldots (n-k+1)$ minors in total, the location of $m$ may be identified as the most frequent location of $m$ in the multiset of $k$-minors when
    \[ \binom{n - m}{k} > \frac{1}{2} \binom{n}{k}. \]
    In particular, if the above inequality is satisfied for $m = k^2 + 2k$, it is satisfied for all $m\le k^2 + 2k$.

    To give an explicit, non-sharp lower bound, observe that
    \[\frac{n-k^2-2k-j}{n-j} = 1-\frac{k^2+2k}{n-j}\geq  1-\frac{k^2+2k}{n-k+1}\]
    for $0\leq j \leq k-1$. So, to find the location of $1, 2, \ldots, k^2 + 2k$, it suffices to have $( 1-\frac{k^2+2k}{n-k+1})^k > \frac{1}{2}$. It is straightforward to rewrite this as
    \begin{align*}
    n & > k-1+\frac{k^2+2k}{1 - 2^{-\frac{1}{k}}}\\ & = k-1 + \frac{2^{\frac{1}{k}}(k^2 + 2k)}{2^{\frac{1}{k}} - 1}= k^2 + 3k - 1+\frac{k^2 + 2k}{2^{\frac{1}{k}} - 1}.
    \end{align*}
    From the Taylor series of the function $f(x)=2^x$ one easily checks
    \[\frac{x}{2^x-1}\le \frac1{\ln 2}-\frac x2 + \frac{\ln 2}{12} x^2\]
    for $x\ge 0$, which gives
    \[k^2 + 3k - 1+\frac{k^2 + 2k}{2^{\frac{1}{k}} - 1}\le \frac{k^3+2k^2}{\ln 2}+\frac {k^2}2 +2k-1 + \frac{\ln 2}{12} (k+2).\qedhere\]
\end{proof}

It is easy to check with Equation \eqref{eq:4.1} that the multiset $\mM_2(T)$ determines $T$ when $n \geq 28$. In fact, we will see below in Theorem \ref{thm: k = 2 sharp} that $n \geq 8$ is the sharp lower bound for determining $T$ by $\M_2(T)$. As a result, the above cubic bound is far from sharp.

\section{Improved Location of $n$}

For $T \in \YT(n)$, Lemma \ref{lemma:location of k2 + 2k through n} shows that $\M_k(T)$ determines the location of $n$ when $n \geq k^2 + 2k + 1$. This section shows the same is true when $n = k^2 + 2k$.

\begin{lemma} \label{lemma: outer corner counting}
    Let $n,k\in\N$ and $T\in\YT(n)$ with OCs $c_1, \ldots, c_\ell$.
    Suppose $|\Out(c_j)| \leq k $ for at most one OC. Then the location of $n$ is recoverable from $\M_k(T)$.
\end{lemma}

\begin{proof}
    Let $x_1, \ldots, x_\ell$ denote the entries of the cells $c_1, \ldots, c_\ell$, respectively.

    Suppose first that $|\Out(c_j)| \geq k + 1$ for all $j$, $1\leq j \leq \ell$. Then by deleting $k$ cells in $\Out(c_j)$, there exists $T'\in \M_k(T)$ in which the cell $c_j$ is a SOC. Find the minimum entry $m_j$ in $c_j$ among all $T'\in\M_k(T)$ in which $c_j$ survives. Since nothing can slide into the cell $c_j$, the entry there in $T'$ equals
    $x_j$ minus the number of deleted cells which were less than $x_j$.

    If $x_j = n$, then $m_j = n - k$ as every other cell is less than $n$.
    If $x_j < n - k$, then certainly $m_j < n - k$.

    If $n - k \leq x_j < n$, only the entries $x_j + 1, x_j + 2, \ldots, n$ are larger than $x_j$ in $\Out(c_j)$. Thus there are at least
    $d_j = k + 1 - (n - x_j)$ entries in $\Out(c_j)$ that are less than $x_j$. Choose any $k$ entries of $\Out(c_j)$ that include at least $d_j$ entries that are less than $x_j$. This results in a $T'$ for which $c_j$ survives and has an entry of at most
    $x_j - d_j = n - k - 1$. Thus $m_j < n - k$. As a result, $x_j$ is $n$ if and only if $m_j = n - k$.

    This allows the location of $n$ to be determined when $|\Out(c_j)| \geq k + 1$ for all $j$. Moreover, if all but one of the $|\Out(c_j)| \geq k + 1$, then the same analysis applies to all but one of the $x_j$. As a result, $n$ may be located if it is one of these $x_j$. If it is not one of these, as $n$ must lie in an OC, it must be the only remaining $x_j$.
\end{proof}

If at least two OCs $c_j$ have $|\Out(c_j)| \leq k$, it turns out that the configuration for $T$ is very limited, at least when $n \geq k^2 +2k$, which is the lower bound for recovering the shape of $T$ from the shapes of $\M_k(T)$, by Equation \eqref{equation: Monks shape determination bound}.

\begin{lemma} \label{lemma: bad shape with small outer area}
    Let $n,k\in\N$ and $T\in\YT(n)$ with OCs $c_1, \ldots, c_\ell$. Suppose $|\Out(c_j)| \leq k$ for at least two OCs. Then $|T| \leq k^2 +2k$. Equality holds if and only if $T$ has shape $((k+1)^k, k)$, \emph{i.e.}, the shape of a $(k+1)\times (k+1)$ square with the lower-right cell removed; see Figure~\ref{figure: k+1 square minus lower right corner}.
\end{lemma}
\begin{figure}[H]
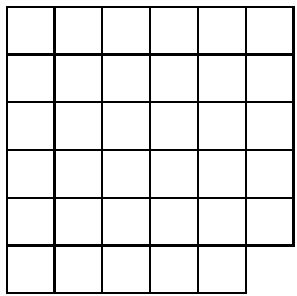

    \centering
    \ydiagram{6,6,6,6,6,5}
    \caption{Largest shape with multiple OCs such that $|\Out(x_j)| \leq k$, where $k=5$}
    \label{figure: k+1 square minus lower right corner}
\end{figure}

\begin{proof}
    Let cells $c, d$ be two OCs of $T$ with outer area at most $k$. We will give an algorithm that modifies $T$ by moving cells and possibly increasing $|T|$ while retaining at least two OCs with outer area at most $k$. The algorithm will end with the shape $((k+1)^k, k)$ to finish the proof.

    We may assume that $c$ is below $d$. Begin by moving all cells of $T$ that are not in either $\In(c)$ or $\In(d)$ below $\In(c)$ in the shape of a rectangle of width $1$, see Figure \ref{figure: move cells pic 1}. Write $T_1$ for the new Young diagram.
    \begin{figure}[H]
    \centering
    \ytableaushort
        {\none, \none\none\none d, \none, \none c, \none }
        * {5,4,3,2,1}
        * [*(lightgray)]{4+1,4+0,2+1,2+0,0+1}
    \;\;\; $\longrightarrow$ \;\;\;\;\;\;\;\;
    \ytableaushort
        {\none, \none\none\none d, \none, \none c, \none, \none, \none }
        * {4,4,2,2,1,1,1}
        * [*(lightgray)]{4+0,4+0,2+0,2+0,0+1,0+1,0+1}
    \caption{$T \longrightarrow T_1$}
    \label{figure: move cells pic 1}
    \end{figure}

    Note that $|T_1|=|T|$. If $c$ is not in the first column, then $c,d$ remain OCs with $|\Out(c)|$ and $|\Out(d)|$ unchanged.
    If $c$ is in the first column, then $d$ remains an OC with $|\Out(d)|$ unchanged. However, the other OC shifts to the bottom of the first column. Its outer area may decrease. In this case, relabel this new OC as $c$.

    Next, let $p$ be the number of cells below $d$ and $q$ be the number of cells to the right of $c$. Form the Young diagram $T_2$ to be of shape
    $((p+1)^q,p)$, see Figure \ref{figure: move cells pic 2}.
    \begin{figure}[H]
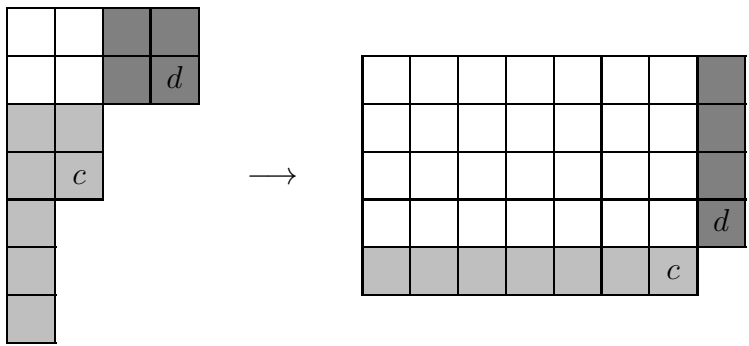

    \centering
    \ytableaushort
        {\none, \none\none\none d, \none, \none c, \none, \none, \none }
        * {4,4,2,2,1,1,1}
        * [*(lightgray)]{4+0,4+0,0+2,0+2,0+1,0+1,0+1}
        * [*(gray)]{2+2,2+2,2+0,2+0,1+0,1+0,1+0}
    $\longrightarrow$ \;\;\;\;\;
    \ytableaushort
        {\none, \none, \none,
        \none\none\none\none\none\none\none d,
        \none\none\none\none\none\none c}
        * {8,8,8,8,7}
        * [*(lightgray)]{8+0,8+0,8+0,8+0,0+7}
        * [*(gray)]{7+1,7+1,7+1,7+1,7+0}
    \caption{$T_1 \longrightarrow T_2$}
    \label{figure: move cells pic 2}
    \end{figure}

    In $T_2$, label the bottom OC $c$ and the top OC $d$. Observe that $|T_2|=pq+p+q \geq |T_1|$,  and that
    $|\Out(c)|$ and $|\Out(d)|$ do not change.

    Finally, look at all possible Young diagrams consisting of a rectangle minus its lower-right cell, whose two outer areas are at most $k$. The maximum area is achieved for $p=q=k$ by the shape $((k+1)^k, k)$.
\end{proof}

We are left with discussing the YT with shape $((k+1)^k, k)$ that arises in Lemma \ref{lemma: bad shape with small outer area}, see Figure \ref{figure: k+1 square minus lower right corner}. We start with an auxiliary result.

\begin{lemma} \label{lemma: movement of n-k in column/row}
    Let $k\in\N$, $n = (k+1)^2 -1$, and $T\in\YT(n)$ with shape $((k+1)^k, k)$. Then the value of each OC is at least $n-k$. The set $\M_k(T)$ determines whether each OC sits in $[n-k,n-2]$, or $[n-1,n]$. Moreover, in the first case, the exact value of the OC is determined.
\end{lemma}

        \begin{figure}[H]
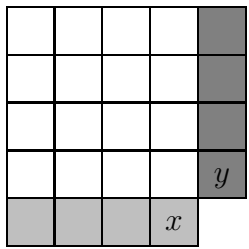

    \ytableaushort
        {\none, \none, \none,
        \none\none\none\none y,
        \none\none\none x}
        * {5,5,5,5,4}
        * [*(lightgray)]{5+0,5+0,5+0,5+0,0+4}
        * [*(gray)]{4+1,4+1,4+1,4+1,4+0}
    \caption{$T$, where $k=4$}
    \label{figure: T where k=4}
    \end{figure}

\begin{proof}
    $T$ has the shape of a $(k+1) \times (k+1)$ square minus its lower-right cell; see Figure \ref{figure: T where k=4}.
    Label the lower OC as $x$ and the upper OC as $y$. Observe that $n$ must be one of them. Write $X$ for the set of cells of $T$ that are in the same row as $x$ and $Y$ for the set of cells of $T$ in the same column as $y$. As the analyses for $x$ and $y$ are similar, we will work only with~$y$.

    The fact that each OC is at least $n-k$ follows from the fact that elements larger than $y$ can only appear in the outer area $\Out(y) =X$, where $|X|=k$.

    First observe that when $y$ is either $n$ or $n-1$, it is possible to find some $T'\in\M_k(T)$ with an $n-k$ in any cell of $Y$ by first deleting $n$ if $y= n-1$ and then deleting upper cells in $Y$ followed by cells in $X$.

    If $y=n-2$, then finding some $T'\in\M_k(T)$ with an $n-k$ in a cell of $Y$ would require some cell of $Y$ to be labeled $n-i$ with $2\leq i \leq k$ at the start of the deletion process, where the label becomes an $n-k$ in $Y$ by the end of the process.
    This requires $k-i$ cells to be deleted which are less than $n-i$ and all $i$ cells to be deleted which are greater than $n-i$, where some cells above $n-i$ may slide out of $Y$ or be deleted. Since the cells larger than $n-i$ will not cause $n-i$ to slide, the farthest $n-i$ can slide and become an $n-k$ is $k-i$ cells up.
    When $i=2$, $n-2$ sits at the bottom of $Y$ and, at most, could result in a label $n-k$ up to cell number $1+(k-2)=k-1$ of $Y$, counting from the bottom. When $i>2$ and $n-i$ sits in the $j$th cell of $Y$, counting from the bottom, then $i\geq 1+j$. In this case, at most, the label $n-k$ could also only reach cell number $j + (k-i) \leq k-1$ of $Y$.
    Moreover, by deleting the appropriate number of cells in $Y$ and $X$, an $n-k$ may be achieved in any of the cells of $Y$ from the bottom up to position $k-1$.

    A similar analysis holds for $y=n-i_0$, $2\leq i_0 \leq k$ except that a label $n-k$ will only appear in any of the cells of $Y$ from the bottom up to position $k-i_0 +1$.
\end{proof}

In addition, for $k\ge 2$, we can show that the location of $n$ in $T$ is determined by the set of $k$-minors $\M_k(T)$.

\begin{theorem} \label{theorem: n position}
    Let $k\in\N_{\ge 2}$, $n = (k+1)^2 -1$, and $T\in\YT(n)$ with shape $((k+1)^k, k)$. Then $\M_k(T)$ determines the location of $n$.
\end{theorem}

\begin{proof}
    By Lemma \ref{lemma: movement of n-k in column/row}, we are done if any of the OCs of $T$ sits in $[n-k,n-2]$.
    Thus, it remains to determine the location of $n$ when $n-1$ and $n$ are the OCs of $T$. Let $\mathcal{L}$ be the set of cells, $c$, satisfying:
    \begin{itemize}
    \item[$(1)$] For all $T'\in\M_k(T)$ for which $c$ exists, its value is at least $n-2k$
    and
    \item[$(2)$] there exist $T'\in\M_k(T)$ for which $c$ has the value $n-2k$.
    \end{itemize}

    First of all, we claim that the cell $c^*$ of $T$ holding $n-k$ lies in $\mathcal{L}$. Condition (1) is automatic here.
    For Condition (2), observe that $c^*$ fails to be an OC when $k\ge 2$. Thus, the size of the complement of $\In(c^*)$ is at least $2k$. Since there are exactly $k$ elements greater than $n-k$, this leaves at least $k$ elements smaller than $n-k$ in the complement of $\In(c^*)$ that can be deleted without moving the cell $c^*$.

        \begin{figure}[H]
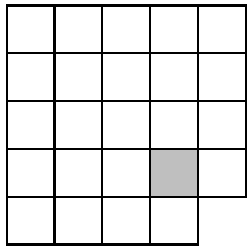

    \ytableaushort
        {\none, \none, \none, \none, \none}
        * {5,5,5,5,4}
        * [*(lightgray)]{5+0,5+0,5+0,3+1,4+0}
    \caption{The cell adjacent to both OCs for $k=4$}
    \label{figure: critical cell}
    \end{figure}

    Next, we claim that if $c$ is not the cell adjacent to both OCs (see Figure \ref{figure: critical cell}), then $c \in \mathcal{L}$ if and only if its value is $n-k$. Suppose $c \in \mathcal{L}$ and $c$ is not adjacent to both OCs. Write $i_0$ for its value.
    If $i_0 > n-k$, then property (2) would be violated. Thus $i_0 \le n-k$. In particular, $c$ is not an OC.
    Next, observe that the size of the complement of $\In(c)$ is at least $2k+1$.
    If $i_0 < n-2k$, deleting any $k$ of the cells in the complement of $\In(c)$ would violate number (1).
    If $i_0 = n-k-i$ for $1\leq i \leq k$, then the complement of $\In(c)$ contains at least $(2k+1)-(k+i) = k + 1 - i$ elements smaller than $i_0$. Deleting $k$ elements from the complement of $\In(c)$, including at least $k+1-i$ that are smaller than $i_0$, reduces the cell's value to at most $n-2k-1$, which violates (1). Therefore, $i_0 = n - k$.

    As a result, it is possible to tell the location $c^*$ of $n-k$ in $T$. If there is a cell in $\mathcal{L}$ that is not adjacent to both OCs, $n-k$ is there. Otherwise, it is in the cell adjacent to both OCs.

    Now pick some $T'\in\M_k(T)$ for which $c^*$ has the value $n-2k$. Any such $T'$ results from $T$ by deleting $k$ cells less than $n-k$ without moving the cell $c^*$. During this deletion process, the entries $n-1$ and $n$ may be moved and are reduced by $k$. However, their relative position is preserved in the following sense.
    If $n-1$ begins in the lower OC of $T$, then, after $i$ deletions, $n-1-i$ will always be in a row weakly below the row of $n-i$. If $n-1$ begins in the upper OC of $T$, then $n-1-i$ will always be in a row strictly above the row of $n-i$. These statements are easily verified by induction.
    For example, if $n-1-i$ is in the same row as $n-i$ after $i$ deletions, then $n-1-i$ must be in the cell immediately to the left of $n-i$. After the next deletions, $n-1-(i+1)$ can only be in a row above $n-(i+1)$ if the jeu de taquin process moves the cell with entry $n-1-i$ up.
    However, in that case, the box above $n-i$ is smaller than $n-1-i$ and would slide over to the left, a contradiction. The other cases are similar.

    As a result, the relative positioning of $n-1-k$ and $n-k$ in $T'$ determines the original positioning of $n-1$ and $n$ in $T$. In particular, $n$ is in the lower OC of $T$ if and only if $n-k$ is in a row strictly below $n-1-k$ in $T'$.
\end{proof}

Combining these results gives the following.

\begin{corollary} \label{cor: location of n}
    Let $n,k\in\N$, $k\ge 2$, and $T\in\YT(n)$.
    Then $\M_k(T)$ determines the location of $n$ when $n \geq k^2 + 2k$.
\end{corollary}

\section{Sharp Bound for $k=2$}
\label{sec:sharp bound k=2}
	
	In this section, we prove that $\M_2(T)$ determines $T$ when $n \geq 8$.
	This result is sharp for $k = 2$ as the mapping $T \longmapsto \M_2(T)$ is not injective when $n = 7$.
	For example, the $2$-minors of the tableaux in Figure \ref{figure: same 2-minors} are identical.
	\begin{figure}[H]
		\centering
		\begin{ytableau}
			1 & 2 & 5& 7 \\
			3 & 4 & 6
		\end{ytableau}
		\;\;\;\;\;\;\;
		\begin{ytableau}
			1 & 3 & 5& 7 \\
			2 & 4 & 6
		\end{ytableau}
		\caption{Tableaux with identical sets $\M_2(T)$}
		\label{figure: same 2-minors}
	\end{figure}
	Of note, the mapping $T \longmapsto \mM_2(T)$ is injective for $n$ equal to $6$ and $7$. This can be verified by a straightforward computer calculation as multisets are easy to distinguish. Injectivity for this mapping fails at $n = 5$.

	\begin{theorem} \label{thm: k = 2 sharp}
		Let $n \in \N$ and $T \in \YT(n)$. Then $\M_2(T)$ determines $T$ when $n \geq 8$.
	\end{theorem}
	
	By Lemma \ref{lemma:location of k2 + 2k through n}, $\M_2(T)$ determines the location of $[9, n]$. Thus, only the location of $[1, 8]$ remains to be determined.
	Lemma \ref{lemma: removal of top elements and Mk} shows that $\M_2(R_{[9, n]} T)$ equals $R_{[7, n-2]} \M_2(T)$. Thus Theorem \ref{thm: k = 2 sharp} follows immediately from the following.
	
	\begin{lemma} \label{lemma: n = 8 lemma for k = 2}
		Let $T \in \YT(8)$. Then $\M_2(T)$ determines $T$.
	\end{lemma}

The proof of Lemma \ref{lemma: n = 8 lemma for k = 2} (which will occupy the rest of this section) breaks naturally into a number of cases according to the shape of $T$, which is known from $M_2(T)$. However, we can significantly reduce the number of cases to check by \cite[Lemma 3.6]{Monks2009}.

\begin{lemma}[\cite{Monks2009}] \label{Monks 3.6}
	Let $n,k\in\N$ and $T\in\YT(n)$. Then $M_2(T)$ determines the shape of $T$ if $n$ cannot be expressed as $n=(a+1)b+c-1$ for $a,b,c\in\N$ satisfying $a\le c\le k$ and $b+ (c\, \mod a) \le k$.
\end{lemma}

Applying Lemma~\ref{Monks 3.6} to the choice $n=6, k=2$ gives that $M_2(R_{[7,8]}T)$ determines the shape of $R_{[7,8]}T$ (hereafter referred to as $T'$). Since 7 and 8 are in the complement of $T'$ and the location of 8 is determined from $M_2(T)$ by Corollary \ref{cor: location of n},  the location of 7 in $T$ is also determined. Thus it only remains to show that $M_2(T)$ determines $T'$. Recall that Lemma  \ref{lemma: removal of top elements and Mk} gives the 2-minors of $T'$ as $M_2(T')=R_{[5,6]}M_2(T)$.
	
	Up to symmetry, there are 6 shapes of $T'$ to consider:
	$(6)$, $(5,1)$, $(4,2)$, $(4,1,1)$, $(3,3)$, and $(3,2,1)$.  

\begin{lemma} \label{hanging shapes}
	Let $T'$ be of shape $(6)$, $(5,1)$, or  $(4,1,1)$. Then $\M_2(T')$ determines $T'$.
\end{lemma}

\begin{proof}
Let $T'$ have top row of length $r$. If $r=6$, there is nothing to prove, as there is only one tableau of shape $(6)$. Otherwise, label the top row entries of $T'$ as $\alpha_1, \ldots, \alpha_{r}$ and look at any $S \in \M_2(T')$ with a top row of length $r-2$. Since nothing can slide up during deletion, the only possible values that could appear in $S$ in the box originally holding $\alpha_i$, for $1 \leq i \leq r - 2$, are $\alpha_i, \alpha_{i+1} -1$, and $\alpha_{i+2} -2$, where $\alpha_i\le \alpha_{i+1} -1 \le \alpha_{i+2} -2$. By deleting the entries $\alpha_{r-1}$ and $\alpha_r$ from $T'$, the minimal value of $\alpha_i$ is always obtained. Therefore $\alpha_i$ is determined by $\M_2(T')$ for all $1 \leq i \leq r - 2$.

If $2 \leq i \leq r - 2$, then deleting $\alpha_{2}$ and $\alpha_{3}$ from $T'$ yields a 2-minor which achieves the maximal value of $\alpha_{i + 2} - 2$ in position $i$ of the first row, determining $\alpha_{i+2}$ to be 2 more than this maximal entry. Therefore, $\alpha_i$ is determined by $\M_2(T')$ also for all $4 \leq i \leq r$.

When $r=5$, i.e., $T'$ is of shape $(5,1)$, the entire top row of $T'$ is now determined. The single remaining entry is also determined, and therefore all of $T'$ is determined from $\M_2(T')$.

When $r=4$, it remains to determine $\alpha_3$. Once $\alpha_3$ is found, the remaining values are known and their positions are determined by strict monotonicity of the first column. If $\alpha_4=\alpha_2+2$,  then $\alpha_3=\alpha_2+1$ is uniquely determined by strict monotonicity across the first row, and therefore $T'$ is determined from $\M_2(T')$.

It remains to consider three cases: $\alpha_2=3$ and $\alpha_4=6$; $\alpha_2=2$ and $\alpha_4=5$; and $\alpha_2=2$ and $\alpha_4=6$.

Let us first consider $\alpha_2=3$ and $\alpha_4=6$. In this case $\alpha_1=1$, the middle row entry is 2, and the bottom row entry is either 4 or 5 (whichever $\alpha_3$ is not). If the bottom row entry is 5, then the 2-minor arising from deleting 6 and 2 will have its 4 in the bottom row. If the bottom row entry of $T'$ is 4, there is no such 2-minor -- such a minor would require deleting the first entry of the top row or the unique entry of the middle row, which would result in shifting the 4 up, but it would then be relabeled as something less than 4. Thus, $\M_2(T')$ determines $T'$ in this case.

Now, let us consider $\alpha_2=2$ and $\alpha_4=5$. Note that if $\alpha_3=3$, the only entry that can occur as the second entry of the top row of a 2-minor of shape $(3,1)$ is a 2, as either the 2 has not moved, or either the 1 or 2 has been deleted, and the 3 has slid over and been relabeled 2. However, if $\alpha_3=4$, the value 3 can and does appear as the entry in this position of the 2-minor arising from deleting $6$ and $2$. Thus $\M_2(T')$ determines all of $T'$.

Finally, we consider $\alpha_2=2$ and $\alpha_4=6$. First recall that the only possible values that can appear as the second entry in the top row of some 2-minor of $T'$ of shape $(2,1,1)$ are $2$, $\alpha_3-1$, and $4$. By deleting various combinations of the top row entries $\alpha_2,\alpha_3,$ and $\alpha_4$, we see that each of these values is achievable. Thus, if there are three distinct entries that occur in this position of a 2-minor, we have $2<\alpha_3-1<4$, so $\alpha_3-1=3$, i.e., $\alpha_3=4$. In this case, $T'$ is determined. Otherwise, there are only two distinct values in that position, and thus either $\alpha_3=3$ or $\alpha_3=5$. Then, by examining possible bottom row entries of a 2-minor of shape $(3,1)$ (as in the case $\alpha_2=3$ and $\alpha_4=6$), we can determine $\alpha_3$. In particular, $4$ appears among the possible entry values in this position if and only if $\alpha_3=3$. Thus again $T'$ is determined.
\end{proof}

\begin{lemma}
	Let $T'$ be of shape  $(3,2,1)$. Then $M_2(T')$ determines $T'$.
\end{lemma}

\begin{proof}
	Note that $T'$ has three OCs, and that the entries in these OCs are $5$ and $6$, and one of $3$ and $4$. Note that only one OC has outer area less than or equal to 2. As such, by Lemma \ref{lemma: outer corner counting}, the location of the 6 is known.
	
	Fix an OC and consider the set of entries that occur in that SOC in a 2-minor. It is easy to see that 4 is a possibility if and only if the OC was originally occupied by $4,5$ or 6.
	
	Thus, if there is an SOC in which no 2-minor has a 4, this cell must be occupied in $T'$ by 3. By strict monotonicity, this cannot be the OC in the middle row, and thus either the first column or top row (depending on which OC is under consideration) is determined, as the other entries in this column/row must be 1 and 2. Since the location of 6 is known, 5 must be in the remaining OC, 4 in the sole remaining cell, and $T'$ is fully determined.

	Otherwise, the OCs of $T'$ are occupied by $4,5,$ and 6, and the entries 2 and 3 must be located in the two cells of $T'$ right next to the 1 in the upper left corner of $T'$. In order to proceed, note that any 2-minor of $T'$ in which 4 remains labeled 4 must preserve every label less than 4 by Lemma \ref{lemmma: construct In(m) when m in same place}. Hence there is exactly one such 2-minor of $T'$ (obtained by deleting the 5 and 6), and all its entries agree with those of $T'$. On the other hand, for every OC of $T'$, deleting the other two OCs results in a 2-minor where this OC is an SOC labeled 4 with every label less than 4 preserved. Thus, there exists at least one OC of $T'$ with a unique 2-minor where this OC is labeled 4, and any such 2-minor will provide us with the correct location of every label less than 4 in $T'$. By symmetry, we may focus now on the case that the label 2 is located in the cell below 1, see Figure \ref{figure: (3,2,1) part 1}. If 2 is located in the cell to the right of 1 instead, just flip the tableau over.

    \begin{figure}[H]
		\centering
    \ytableaushort
        {1 3 \none, 2 \none,
        \none}
        * {3,2,1}
		\caption{The position of $1,2,$ and 3 in $T'$}
		\label{figure: (3,2,1) part 1}
	\end{figure}

    Since the location of 6 is known, it remains to locate the entries 4 and 5 in $T'$. We are left with considering three pairs of tableaux as depicted in Figure \ref{figure: (3,2,1) part 2}. Each pair corresponds to a different location of the entry 6. In each case we need to distinguish between the left and right tableau on basis of their 2-minors.

    \begin{figure}[H]
		\centering
      \ytableaushort
        {1 3 5, 2 4,
        6}
        * {3,2,1}
        * [*(lightgray)]{2+1,1+1,0+0}
		\;\;\;\;\;\;\;
    \ytableaushort
        {1 3 4, 2 5,
        6}
        * {3,2,1}
        * [*(lightgray)]{2+1,1+1,0+0}
  \\[.5cm]
  \par\noindent\rule{\textwidth}{0.4pt}
  \\[.5cm]
    \ytableaushort
        {1 3 5, 2 6,
        4}
        * {3,2,1}
        * [*(lightgray)]{2+1,2+0,0+1}
		\;\;\;\;\;\;\;
    \ytableaushort
        {1 3 4, 2 6,
        5}
        * {3,2,1}
        * [*(lightgray)]{2+1,2+0,0+1}
  \\[.5cm]
  \par\noindent\rule{\textwidth}{0.4pt}
  \\[.5cm]
      \ytableaushort
        {1 3 6, 2 4,
        5}
        * {3,2,1}
        * [*(lightgray)]{2+0,1+1,0+1}
		\;\;\;\;\;\;\;
      \ytableaushort
        {1 3 6, 2 5,
        4}
        * {3,2,1}
        * [*(lightgray)]{2+0,1+1,0+1}
		\caption{The remaining options for the tableau $T'$}
		\label{figure: (3,2,1) part 2}
	\end{figure}

    We start with the top pair in Figure \ref{figure: (3,2,1) part 2}. Note that the left tableau has a 2-minor of shape $(3,1)$ with 3 as the bottom row entry (obtained by deleting the 2 and 6). The right tableau has no such 2-minor. In particular, note that 4 would be an SOC for such a 2-minor. Thus, Lemma \ref{lemmma: construct In(m) when m in same place} applies and the 2-minor must preserve every label less than~4, a contradiction.

    For the other two pairs, similar arguments apply. For the middle pair, the left
    tableau has again a 2-minor of shape $(3,1)$ with 3 as the bottom row entry (obtained by deleting the 2 and 6). For the bottom pair, the left tableau has a 2-minor of shape $(2,1,1)$ with 2 as the right column entry (obtained by deleting the 2 and 6).
	%
 \end{proof}

It is worth pointing out that, in fact, $M_2(T')$ alone is enough to determine $T'$ in the above two lemmas, establishing that there are many elements of $YT(6)$ which are recoverable from their 2-minors. In the remaining lemmas, information about $M_2(T')$ is not sufficient, and we will make use of the additional information provided by $M_2(T)$.

\begin{lemma} \label{3-3}
	Let $T'$ be of shape $(3,3)$. Then $M_2(T)$ determines $T'$.
\end{lemma}

\begin{proof}
	
	First, consider the case that $T$ is of shape $(5,3)$. Label the first three entries of the top row as $\alpha_1=1$, $\alpha_2$, and $\alpha_3$. By the argument presented in the proof of Lemma \ref{hanging shapes}, $\alpha_2$ and $\alpha_3$ are determined to be the minimal values occurring as the corresponding entries of the top row of a 2-minor of $T$ of shape $(3,3)$. Thus, the top row of $T'$ is known, and since $T'$ has exactly two rows, $T'$ is completely determined by strict monotonicity.
	
	Now, suppose that $T$ is \emph{not} of shape $(5,3)$, i.e., the top row of $T$ has at most 4 entries.	This leaves us with the options $(4,4)$, $(4,3,1)$, $(3,3,2)$, and $(3,3,1,1)$ for the possible shape of $T$, see Figure \ref{figure: 4 tableaux}.

	\begin{figure}[H]
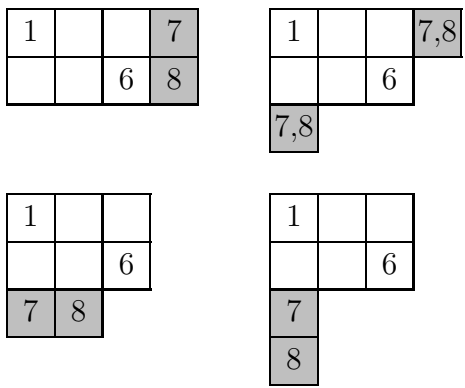

		\centering
    \ytableaushort
        {1 \none\none 7,
        \none\none 6 8, \none}
        * {4,4,0}
        * [*(lightgray)]{3+1,3+1}
		\;\;\;\;\;\;\;
    \ytableaushort
        {1 \none\none {7,\!8}, \none\none 6,
        {7,\!8}}
        * {4,3,1}
        * [*(lightgray)]{3+1,3+0,0+1}
  \\[.5cm]
    \ytableaushort
        {1 \none\none, \none\none 6,
        7 8, \none}
        * {3,3,2,0}
        * [*(lightgray)]{3+0,3+0,0+2}
		\;\;\;\;\;\;\;
    \ytableaushort
        {1 \none\none, \none\none 6, 7,
        8}
        * {3,3,1,1}
        * [*(lightgray)]{3+0,3+0,0+1,0+1}
		\caption{The remaining options for the tableau $T$}
		\label{figure: 4 tableaux}
	\end{figure}
	
	Let $\mathcal L$ be the set of all $S\in M_2(T)$ in which the position of $7$ in $T$ is now occupied by $6$. We will show that  $\{R_{6}S:S\in \mathcal L\}=M_1(T')$. Since $T'\in YT(6)$, this will determine $T'$ by \cite{Monks2009}.
	
	First, note that the deletion of 8 followed by the deletion of some $i$ less than $7$ cannot result in 7 sliding, and therefore results in $7$ being relabeled $6$, i.e., results in a 2-minor $S$ belonging to $\mathcal L$. It is immediate that $R_{6}S$ is precisely the $1$-minor of $T'$ resulting from deletion of~$i$. Thus $\{R_{6}S:S\in \mathcal L\}\supseteq M_1(T')$.
	
	To see the reverse inclusion, let us first observe that there are three possible ways in which a 2-minor of $T$ can belong to $\mathcal L$:
 \begin{itemize}
 \item[$(a)$] 7 does not slide, and is relabeled 6;
 \item[$(b)$] 7 is deleted, leading 8 to slide into its position and be relabeled 6; or
 \item[$(c)$] 7 slides, leading 8 to slide into its position and be relabeled 6.
 \end{itemize}
	
	It is clear that any 2-minor arising in fashion $(b)$ can be realized by deleting 8 rather than 7, and thus arises in fashion $(a)$ as well. Any 2-minor $S$ arising in fashion $(a)$ is the result of deleting 8 and some $i$ less than 7. In this case $R_{6}S$ is precisely the 1-minor of $T'$ given by deletion of $i$, i.e., $R_{6}S\in M_1(T')$.
	
	We complete the proof by demonstrating that there are no members of $\mathcal L$ arising in fashion $(c)$. Suppose that $S$ were such a 2-minor.
	
	Since 7 slides in the creation of $S$, it cannot be in the third row of $T$ from the top, as it would require the deletion of at least three entries for it to slide. Thus 7 is the fourth entry of the top row of $T$. Since 8 slides into the position of 7 in the creation of $S$, $T$ must be of shape $(4,4)$ with 8 as the fourth bottom row entry.
	
	Now, fix $i<7$ and consider the result of deleting $i$ from $T$. It is evident that this culminates in either 6 sliding or being deleted, leading 8 to slide left. It is clear that 8 cannot then slide into the position occupied by 7, a contradiction.	
\end{proof}

Using arguments similar to those above, we are able to prove the following, which will complete the proof of  Lemma  \ref{lemma: n = 8 lemma for k = 2}.

\begin{lemma}
	Let $T'$ be of shape $(4,2)$. Then $M_2(T)$ determines $T'$.
\end{lemma}

\begin{proof}
	Label the entries in the top row of $T'$ as $\alpha_1=1,\alpha_2,\alpha_3,$ and $\alpha_4$. As in the proof of Lemma \ref{hanging shapes}, $\alpha_2$ is determined by investigating 2-minors in $\M_2(T')$ of shape $(2,2)$. Note that if we are able to determine $\alpha_3$ and $\alpha_4$, strict monotonicity along the bottom row will determine the remainder of $T'$.
	
	If $\alpha_2=4$, the entries $\alpha_3=5$ and $\alpha_4=6$ are determined by strict monotonicity, and therefore all of $T'$ is determined.
	
	If $\alpha_2=3$, then we know that $\{\alpha_3,\alpha_4\}\subseteq\{4,5,6\}$. If $4\notin\{\alpha_3,\alpha_4\}$, the bottom row of $T'$ has entries $2$ and $4$. If $S\in M_2(T')$ is of shape $(3,1)$, then the bottom row entry of $S$ is at most 3 (as 4 will be relabeled if it has slid into this position). If $4\in\{\alpha_3,\alpha_4\}$, then the bottom row of $T'$ has entries 2 and either 5 or 6. Either way, there is some $S\in M_2(T')$ of shape $(3,1)$ with 4 as bottom row entry (achieved by deleting both 2 and $\alpha_4$). Thus, $M_2(T')$ determines whether $4\in \{\alpha_3,\alpha_4\}$. If $4\notin\{\alpha_3,\alpha_4\}$, then $\alpha_3=5$, $\alpha_4=6$, and $T'$ is thus fully determined. Otherwise $4\in\{\alpha_3,\alpha_4\}$, and in particular, $\alpha_3=4$. We determine $\alpha_4$ as follows. If $\alpha_4=5$, then the top row of $T'$ is $(1,3,4,5)$. If $S\in M_2(T')$ is of shape $(3,1)$, it must arise from deletion of at most one of $3,4,$ or 5 (and at least one of $1,2,$ or 6), which can only result in a top row of $S$ equal to $(1,2,3)$ or $(1,3,4)$. However, if $\alpha_4=6$, there is a 2-minor of $T'$ of shape $(3,1)$ with top row $(1,2,4)$, which is achieved by deleting $4$ and $2$. Thus, $M_2(T')$ distinguishes these cases as well, and $T'$ is determined.
	
	Finally, we consider the case when $\alpha_2=2$. It is here that we will need the extra information given by $M_2(T)$. For the rest of the proof, we will label the top row entries of $T'$ as $\alpha_1 =1, \alpha_2 =2, \alpha_3,$ and $\alpha_4$.
	
	We begin by considering the case that 7 is the 5th entry in the top row of $R_8T$; see Figure \ref{figure: 7 in T part 1}. Applying an argument similar to that we used in Lemma \ref{hanging shapes} for shape $(4,1,1)$, we can determine $R_8T$ from $M_2(R_8T)$. Specifically, we can determine $\alpha_3$ as the minimal entry that occurs in this position in the top row of 2-minors of $R_8T$ of shape $(3,2)$. Once $\alpha_4$ is known, the second row of $R_8T$ is determined by strict monotonicity. The entry $\alpha_4$ can be found as follows: If there is no 2-minor of $R_8T$ of shape $(4,1)$ with
    5 as the second row entry, then $\alpha_4 =6$. If $\alpha_4 \ne 6$ and there exists some 2-minor of $R_8T$ of shape $(3,2)$ with 4 as third top row entry, then $\alpha_4 =5$. Otherwise $\alpha_4 =4$.

 	\begin{figure}[H]
		\centering
    \ytableaushort
        {1 2 {\alpha\textsubscript{3}} {\alpha\textsubscript{4}} 7,
        \none}
        * {5,2}
        * [*(lightgray)]{4+1,2+0}
		\caption{The position of 7 in the top row of $R_8 T$}
		\label{figure: 7 in T part 1}
	\end{figure}

	Now, suppose that $7$ is in the third row of $R_8T$ from the top; see Figure~\ref{figure: 7 in T part 2}. In this case, looking at the third top row entry of 2-minors of $R_8T$ of shape $(3,2)$, the minimal available value gives $\alpha_3$ and the maximal value gives $\alpha_4 -1$. In particular, note that 7 cannot slide up with only one single deletion of something less than 7 and must itself be deleted instead.

   	\begin{figure}[H]
		\centering
    \ytableaushort
        {1 2 {\alpha\textsubscript{3}} {\alpha\textsubscript{4}}, \none,
        7}
        * {4,2,1}
        * [*(lightgray)]{4+0,2+0,0+1}
		\caption{The position of 7 in the third row of $R_8 T$}
		\label{figure: 7 in T part 2}
	\end{figure}
	
	Finally, suppose that $7$ is in the second row of $R_8T$; see Figure \ref{figure: 7 in T part 3}. We start with some basic observations on $\alpha_4$. By strict monotonicity along the top row, $\alpha_4$ is either $4,5,$ or~6.

    \begin{figure}[H]
		\centering
    \ytableaushort
        {1 2 {\alpha\textsubscript{3}} {\alpha\textsubscript{4}},
        \none\none 7}
        * {4,3}
        * [*(lightgray)]{4+0,2+1}
		\caption{The position of 7 in the second row of $R_8 T$}
		\label{figure: 7 in T part 3}
	\end{figure}

	 Consider the set of 2-minors of $R_8T$ of shape $(4,1)$. If $\alpha_4$ is 5 or 6, there is at least one such 2-minor with 5 as the last entry of its top row (obtained by deleting 7 and the second entry of the bottom row). If $\alpha_4=4$, then no such 2-minor is present. In this case, monotonicity along rows forces $\alpha_3=3$, completing the determination of $T'$.

  Thus, the case $\alpha_4 \in \{4,5\}$ remains. This leaves us with the five tableaux in Figure \ref{figure: 5 tableaux}.


	\begin{figure}[H]
		\centering
    \ytableaushort
        {1 2 5 6,
        3 4 7}
        * {4,3}
        * [*(lightgray)]{4+0,2+1}
		\;\;\;\;\;\;\;
    \ytableaushort
        {1 2 4 6,
        3 5 7}
        * {4,3}
        * [*(lightgray)]{4+0,2+1}
    		\;\;\;\;\;\;\;
    \ytableaushort
        {1 2 4 5,
        3 6 7}
        * {4,3}
        * [*(lightgray)]{4+0,2+1}
  \\[.5cm]
    \ytableaushort
        {1 2 3 6,
        4 5 7}
        * {4,3}
        * [*(lightgray)]{4+0,2+1}
		\;\;\;\;\;\;\;
    \ytableaushort
        {1 2 3 5,
        4 6 7}
        * {4,3}
        * [*(lightgray)]{4+0,2+1}
        \;\;\;\;\;\;\;\;\;\;\;\;\;\;\;\;\;\;\;\;\;\;\;\;\;\;\;\;\;\;\;\;\;\;\;\;
		\caption{The remaining options for the tableau $R_8T$}
		\label{figure: 5 tableaux}
	\end{figure}

 The three tableaux on the top row of Figure \ref{figure: 5 tableaux} can now be identified by investigating the possible values of the bottom row entry of their 2-minors of shape $(4,1)$. First, note that for the top left tableau the value 2 is possible in this position (obtained by deleting the 7 and 2 of $R_8T$). For the remaining four tableaux, cells can only slide up after a top row cell has been deleted first during an earlier jeu de taquin process, and the shape $(4,1)$ must result from deleting two of the bottom row entries of $R_8T$. In particular, none of these tableaux will have a 2-minor of shape $(4,1)$ with bottom row entry 2. This distinguishes the top left tableau in Figure \ref{figure: 5 tableaux} from the other four tableaux. For the remaining four tableaux, one easily verifies that all of the values $3,4,$ and $5$ are possible as bottom row entries of 2-minors of shape $(4,1)$ if and only if the top row of $R_8T$ is $(1,2,4,6)$. If only values 3 and 5 appear in this position, then the top row of $R_8T$ is $(1,2,4,5)$. And if only values 4 and 5 appear, then $R_8T$ is one of the two tableaux on the bottom row of Figure \ref{figure: 5 tableaux}.

 It remains to distinguish between the two tableaux on the bottom row of Figure \ref{figure: 5 tableaux}. Note that the bottom left tableau has a 2-minor of shape $(3,2)$ with top row $(1,3,4)$ (obtained by deleting the 1 and 2 of $R_8T$) while direct inspection shows that no such 2-minor exists for the bottom right tableau. In particular, it would be necessary to delete two of the entries $1,2,$ and $3$ of $R_8T$ to achieve a 2-minor of shape $(3,2)$ without the value 2 as a top row entry. However, every 2-minor resulting in this fashion will have top row $(1,3,5)$.
    %
\end{proof}

\section{Concluding Remarks}

For reconstructibility from $\M_k(T)$ for general $k$, it seems that the most one can hope for is the following quadratic bound, which is suggested by the few cases we are able to check via software:

\begin{conjecture}
    \label{conj:k^2 + 2k}
    Let $k \geq 2$.  Let $T \in \YT(n)$, where $n \geq k^2 + 2k$.  Then $\M_k(T)$ determines $T$.
\end{conjecture}

As for reconstructibility from multisets $\mM_k(T)$, there is some evidence that the lower bound is linear in $k$.  We have been able to verify the following conjecture by computer for $k \leq 5$:

\begin{conjecture}
\label{conj:k+4}
    Let $k \geq 1$.  Let $T \in \YT(n)$, where $n \geq k + 4$.  Then $\mM_k(T)$ determines $T$.
\end{conjecture}

As a final remark, we point out that the lower bound on $n$ for reconstructibility depends very much on the shape of tableaux: upon restriction to certain shapes, $\M_k(T)$ determines $T$ for values of $n$ that are much smaller than the general lower bound.  This can be used to streamline proofs in the course of solving the problem for higher values of $k$.



\bibliographystyle{alpha}
\bibliography{main}

\end{document}